\documentclass[a4paper,11pt,leqno]{amsart}

\usepackage{amssymb,hyperref,latexsym,amsmath}

\theoremstyle{plain}

\newtheorem{theo}{Theorem}[section]

\newtheorem{cor}[theo]{Corollary}

\newtheorem{proposition}[theo]{Proposition}

\newtheorem{theorem}[theo]{Theorem}
\newtheorem{lemma}[theo]{Lemma}
\newtheorem{remark}[theo]{Remark}

\theoremstyle{definition}
\newtheorem{defi}[theo]{Definition}

\theoremstyle{remark}

 \def\RR{{\mathbb R}}

\def\NN{{\mathbb N}}

\begin{document}

\title[Geometric structures on surfaces]{Locally homogeneous rigid geometric structures on surfaces}

\author[S. Dumitrescu]{Sorin DUMITRESCU$^\star$}

\address{${}^\star$ D\'epartement de Math\'ematiques d'Orsay, 
Bat. 425, U.M.R.   8628  C.N.R.S.,
Univ. Paris-Sud (11),
91405 Orsay Cedex, France}
\email{Sorin.Dumitrescu@math.u-psud.fr}
 
 \thanks{This work was  partially supported by  the ANR Grant Symplexe BLAN 06-3-137237}
\date{\today}

\setcounter{tocdepth}{1}

\maketitle

\begin{abstract}    We study locally homogeneous rigid geometric structures on surfaces. We show that a locally homogeneous projective connection on a compact surface is flat.
We also show that a locally homogeneous unimodular  affine connection $\nabla$ on a two dimensional torus is complete and, up to a  finite cover,   homogeneous. 

Let $\nabla$ be a  unimodular real analytic affine connection on a real analytic compact connected surface $M$. If $\nabla$ is   locally homogeneous on a nontrivial open set in $M$, we prove that $\nabla$ is locally homogeneous on all of $M$.\\ 
\end{abstract}

\section{Introduction}

Riemannian metrics  are  the most commun (rigid) geometric structures. A locally homogeneous Riemannian metric on a surface has constant sectional curvature and
it is locally isometric  either to the standard metric on the  two-sphere, or to the flat metric on $\RR^2$, or to the hyperbolic metric on the Poincar\'e's upper-half plane (hyperbolic plane).
Obviously the flat metric is translation invariant on $\RR^2$. Recall also that  the isometry group  of the hyperbolic plane is isomorphic to $PSL(2,\RR)$ and contains  copies of the affine group
 of the real line (preserving orientation)  $Aff(\RR)$,  which act locally free (they are the   stabilizers  of   points on the boundary). Consequently,  the hyperbolic plane is locally isometric to a translation invariant Riemannian metric on $Aff(\RR)$.

We generalize here this phenomena to all  rigid geometric structures in Gromov's sense (see the definition in the following section).

\begin{theo}  \label{surfaces} Let $\phi$ be a  locally homogeneous rigid geometric structure   on a surface. Then
$\phi$ is locally isomorphic to a rigid geometric structure  which  is either rotation invariant on the two-sphere, or translation invariant on $\RR^2$,  or  translation invariant on the  affine group  of the real line preserving orientation $Aff(\RR)$.
\end{theo}

If  $\phi$ is an affine connection this was first proved  by B. Opozda~\cite{Opozda} for the case torsion-free (see also the group-theoretical approach in~\cite{KOV}), and then
 by T. Arias-Marco and O. Kowalski in the case of arbitrary torsion~\cite{AM-K}.

  Theorem~\ref{surfaces} stands, in particular, for projective connections. In this case the theorem can also be deduced  from the results obtained in~\cite{BMM}. Recall that a well-known class  of locally homogeneous  projective connections $\phi$  on  surfaces  are those which are  {\it flat}, i.e.  locally isomorphic to the standard projective connection of the projective plane  $P^2(\RR)$. The automorphism  group of $P^2(\RR)$ is the projective group $PGL(3,\RR)$ which  contains $\RR^2$ acting  freely (by translation) on the affine plane: it is the subgroup which fixes each point in  the line at the infinity. Consequently, $\phi$ is locally isomorphic  to a translation invariant projective  connection on $\RR^2$.
 
 We prove also the following global  results dealing with projective and affine connections on surfaces:
 
 \begin{theorem} \label{projective}  A locally homogeneous projective connection on a compact surface is flat.
 \end{theorem}

  \begin{theorem} \label{affine connection1} Let $M$ be a compact connected real analytic  surface endowed with a unimodular real  analytic   affine connection $\nabla$. If $\nabla$ is locally homogeneous on a  nontrivial  open set in $M$, then $\nabla$ is locally homogeneous on  $M$.
 \end{theorem}
 
 The main result of the article is Theorem~\ref{affine connection1}. It is motivated by the celebrated open-dense orbit theorem of M. Gromov~\cite{DG,Gro} (see also~\cite{Benoist2,CQ,Feres}). Gromov's result asserts that a rigid geometric structure admitting an automorphism group which acts  with a dense orbit is locally homogeneous on an open dense set.  This  maximal  locally homogeneous open (dense)  set appears to be  mysterious and it might very well happen that it coincides with all of the (connected)  manifold in many interesting geometric backgrounds. This was proved, for instance,  for Anosov flows with differentiable stable and instable foliations and  transverse contact  structure~\cite{BFL} and for
 three dimensional compact Lorentz manifolds admitting a nonproper one parameter group acting by automorphisms~\cite{Zeghib}. In ~\cite{BF}, the authors deal with this question  and their results  indicate ways in which the rigid geometric structure cannot degenerate off the open dense set.

 Surprisingly, the extension of a locally homogeneous open dense subset to all of the (connected) manifold  might stand {\it even without assuming the existence of a big automorphism group}. This  is known to be true in the Riemannian setting~\cite{Tri}, as a consequence of the fact that all scalar invariants are constant (see also Corollary~\ref{metric} in our section~\ref{section4}). This was also  recently  proved in the frame of  three dimensional real analytic Lorentz metrics~\cite{Dumitrescu} and for 
 complete  real analytic pseudo-Riemannian metrics~\cite{Melnick}. 
 
 Theorem~\ref{affine connection1}  proves the extension  phenomenon in the setting of  affine connections on surfaces. We don't know  if the  result   is still true when $\nabla$ is not   unimodular analytic, or in higher dimension.
 
 As a by-product of the proof we get  the following:
 
 \begin{theorem} \label{sur tore} A locally homogeneous unimodular affine connection on a two dimensional torus is complete and, up to a finite cover,  homogeneous.
 \end{theorem}
 
 The composition of the article is the following. In section~\ref{section2} we introduce the basic facts about rigid geometric structures and prove theorem~\ref{surfaces}. Section~\ref{section3} deals with global phenomena and proves theorem~\ref{projective}. Theorems~\ref{affine connection1}  and~\ref{sur tore} will be proved in sections~\ref{section4}
 and~\ref{section 5} respectively.

 \section{Locally homogeneous rigid geometric structures}  \label{section2}
 
 In the sequel all  manifolds will be  supposed to be smooth and connected. The geometric structures will be  also assumed to be smooth.

Consider  a  $n$-manifold $M$ and, for all integers $r \geq 1$,   consider the associated bundle $R^r(M)$ of $r$-frames, which is a $D^r(\RR^n)$-principal bundle  over $M$, with $D^r(\RR^n)$ the real algebraic group of $r$-jets at the origin of local diffeomorphisms  of  $\RR^n$ fixing $0$ (see ~\cite{AVL}).

Let us consider, as in~\cite{DG,Gro}, the following 

\begin{defi} A {\it  geometric structure} (of order $r$)  $\phi$ on a   $M$ is a $D^r(\RR^n)$-equivariant smooth   map from  $R^r(M)$ to a real algebraic
variety $Z$ endowed with an algebraic action of $D^r(\RR^n)$.
\end{defi}

Riemannian and pseudo-Riemannian metrics, affine and projective connections and the most encountered geometric objects in differential geometry are known
to verify the previous definition~\cite{DG,Gro, Benoist2, CQ, Feres}. For instance, if the image of $\phi$ in $Z$ is exactly one orbit, this orbit identifies with a  homogeneous space  $D^r(\RR^n)/G$, where $G$ is the stabilizer
of a chosen point in the image of $\phi$. We get then a reduction  of the structure group of $R^r(M)$  to the subgroup $G$. This is exactly the classical definition of a $G$-structure (of order $r$): the case  $r=1$ and $G=O(n,\RR)$
corresponds to a Riemannian metric and that of  $r=2$ and $G=GL(n,\RR)$ gives a torsion free affine connection~\cite{Kobayashi,AVL}.

\begin{defi}  A (local) Killing field of $\phi$ is a (local) vector field on $M$ whose canonical lift to $R^r(M)$ preserves $\phi$.
\end{defi}

Following Gromov~\cite{Gro,DG} we define rigidity as:

\begin{defi} A geometric structure $\phi$ is rigid  at order $k \in \NN$, if  local Killing fields are determined by their   $k$-order jet at  any chosen  point in $M$.
\end{defi}

Consequently, in the neighborhood  of  any point of $M$, the algebra of Killing fields  of  a rigid geometric  structure is finite dimensional.

Recall that (pseudo)-Riemannian metrics, as well as affine and projective connections, or conformal structures in dimension $\geq 3$ are known to be rigid~\cite{DG,Gro, Benoist2, CQ, Feres, Kobayashi}.

\begin{defi} The geometric structure $\phi$ is said to be locally homogeneous on the open subset $U \subset M$ if for any tangent vector $V  \in T_{u}U$ there exists a local Killing field $X$ of
$\phi$ such that $X(u)=V$.
\end{defi}

The Lie algebra of Killing fields is the same at the neighborhood of any point of a locally homogeneous geometric structure $\phi$. In this case it will be simply called {\it the Killing
algebra of $\phi$.}

Let $G$ be a connected Lie group and $I$ a closed subgroup of $G$. Recall  that $M$ is said to be {\it locally modelled} on the homogeneous space $G/I$ if it admits
an atlas with open sets diffeomorphic to open sets in $G/I$ such that the transition maps are given by restrictions of elements in $G$.

In this situation  any $G$-invariant geometric structure $\tilde{\phi}$ on $G/I$ uniquely defines a locally homogeneous geometric structure $\phi$ on $M$ which is locally isomorphic to $\tilde{\phi}$. 

We recall that there exists locally homogeneous Riemannian metrics on $5$-dimensional manifolds which are not locally isometric to a  invariant Riemannian metric on a homogeneous space~\cite{K,LT}. However this phenomenon cannot happen
in lower dimension:

 \begin{theorem} \label{model} Let $M$ be a manifold of dimension $\leq 4$  bearing a locally homogeneous rigid geometric structure $\phi$ with Killing algebra $\mathfrak{g}$. Then $M$ is locally modelled on a homogeneous space $G/I$, where $G$ is a connected Lie group with Lie algebra $\mathfrak{g}$ and $I$ is a closed subgroup of $G$. Moreover, $(M, \phi)$ is locally isomorphic to a $G$-invariant geometric structure on $G/I$.
 \end{theorem}
 
  \begin{proof} Let $\mathfrak{g}$ be the Killing algebra of $\phi$. Denote by $\mathfrak{I}$ the (isotropy) subalgebra of $\mathfrak{g}$ composed by  Killing fields vanishing at a given point in $M$. Let $G$ be the unique connected simply connected  Lie group
 with   Lie algebra  $\mathfrak{g}$. Since $\mathfrak{I}$ is of codimension $\leq 4$ in $\mathfrak{g}$, a result of Mostow~\cite{Mos} (chapter 5, page 614)  shows that  the Lie subgroup $I$  in  $G$ associated to $\mathfrak{I}$ is  {\it closed}. Then $\phi$ induces a $G$-invariant geometric structure $\tilde{\phi}$ on $G/I$ locally isomorphic to it. Moreover, $M$ is locally modelled on $G/I$. 
 \end{proof}
 
 \begin{remark} \label{model simply connected} By  the previous construction, $G$ is simply connected and $I$ is connected (and closed), which implies that $G/I$ is simply connected (see~\cite{Mos}, page 617,
 Corollary 1). In general, the $G$-action on $G/I$ admits a nontrivial discrete kernel. We can  assume that this action is effective  considering the quotient of $G$ and $I$ by the maximal normal subgroup of $G$ contained in $I$ (see proposition 3.1 in~\cite{Sharpe}).
 \end{remark}
 
 Theorem~\ref{surfaces} is a direct consequence of theorem~\ref{model} and of the following:
 
 \begin{proposition} A two dimensional homogeneous space $G/I$ of a connected simply connected Lie group $G$ is either the two  sphere with $G$ being $S^3$, or it bears an  action of a  two dimensional  subgroup of $G$  (isomorphic either to $\RR^2$, or to the affine group of the real line preserving orientation) which admits an open orbit.
 \end{proposition}
 
 Here the $3$-sphere $S^3$ is endowed with its standard structure of Lie group~\cite{Kir, Olver}.
 
  \begin{proof} This follows directly from Lie's classification of the two dimensional homogeneous space
  (see the list  in~\cite{Olver} or~\cite{Mos}). More precisely, any (finite dimensional)  Lie algebra acting transitively on a surface either  admits a two dimensional subalgebra
  acting simply transitively (or equivalently, which trivially intersects the isotropy subalgebra), or it is isomorphic to the Lie algebra of $S^3$ acting on the real sphere $S^2$
  by the standard action.
  \end{proof}
  
  \section{Global rigidity results}  \label{section3}
  
  Recall that a manifold $M$ locally modelled on a homogeneous space $G/I$ gives rise to a {\it developing map} defined on its universal cover $\tilde M$
  with values in $G/I$ and to a {\it holonomy morphism} $\rho : \pi_{1}(M) \to G$ (well defined up to conjugacy in $G$)~\cite{Sharpe}. The developing map is a  local diffeomorphism which is equivariant with respect to the action of the fundamental group
  $\pi_{1}(M)$ on $\tilde M$ (by deck transformations) and on $G$ (by its image through $\rho$).
  
    The manifold $M$ is said to be {\it complete} if the developing map is a global diffeomorphism. In this case $M$ is diffeomorphic to a quotient of $G/I$ by a discrete subgroup
  of   $G$ acting properly and without fixed points.
    
    We give now a last definition:
    
    \begin{defi} A geometric structure $\phi$  on $M$ is said to be of {\it Riemannian type} if there exists a Riemannian metric on $M$ preserved by all Killing fields of $\phi$.
    \end{defi}
    
    Roughly speaking a locally homogeneous geometric structure is of Riemannian type if it is constructed by putting together a Riemannian metric and any other geometric structure (e.g. a vector field).
    Since Riemannian metrics are rigid, a geometric structure of Riemannian type it is automaticaly rigid.
    
    With this terminology we have the following corollary of  theorem~\ref{model}.
    
    \begin{theorem}  \label{isotropie compacte} Let  $M$ be a compact  manifold  of dimension $\leq 4$  equipped  with a locally homogeneous  geometric structure $\phi$ of Riemannian type. Then $M$ is isomorphic to a quotient of a  homogeneous space $G/I$, endowed with a $G$-invariant geometric structure, by a lattice in $G$.
 \end{theorem}
 
 \begin{proof} By theorem~\ref{model},  $M$ is locally modelled on a homogeneous space $G/I$. Since $\phi$ is of Riemannian type, $G/I$ admits a $G$-invariant Riemannian metric.
 This implies that the isotropy $I$ is compact. 
 
 On the other hand, compact manifolds locally modelled on homogeneous space $G/I$ with compact isotropy group $I$ are  classically known to be complete: this is a consequence of the
  Hopf-Rinow's geodesical completeness~\cite{Sharpe}.
  \end{proof}
  
  \begin{remark} A $G$-invariant geometric structure on $G/I$ is of Riemannian type if and only if $I$ is compact.
  \end{remark}

  Recall that a homogeneous space $G/I$ is said to be {\it imprimitive} if the canonical  $G$-action preserves a non trivial foliation.
  
    \begin{proposition}  \label{imprimitive} If $M$ is a compact surface locally modelled on an imprimitive homogeneous space, then $M$ is a torus.
 \end{proposition}
 
 \begin{proof} The $G$-invariant  one dimensional foliation on $G/H$ descends on $M$ to  a non singular foliation. Hopf-Poincar\'e's   theorem implies then that the
 genus of $M$ equals one: $M$ is a torus.
 \end{proof}
 
 Note  that the results of~\cite{KOV,AM-K} imply in particular:
 
 \begin{theorem} \label{imprimitive connections} A locally homogeneous affine connection  on a surface  which is neither  torsion free and flat, nor  of Riemannian type,  is locally modelled on an imprimitive homogeneous space.
 \end{theorem}
 
 Indeed, T. Arias-Marco and O. Kowalski study in~\cite{AM-K} all possible local normal forms for locally homogeneous affine connections on surfaces  with  the corresponding Killing algebra. Their results  are summarized in a nice table (see ~\cite{AM-K}, pages 3-5). In all cases, except for the Killing algebra of the (standard) torsion free affine connection and for Levi-Civita connections
 of Riemannian metrics of constant sectional curvature, there exists at least one   Killing field  non contained in the isotropy algebra which is normalized by the Killing algebra. Its direction defines then a $G$-invariant line field on $G/I$.

The previous result combined with  proposition~\ref{imprimitive}  imply  the main result in~\cite{Opozda}:

\begin{theorem} (Opozda) \label{Opozda} A compact surface $M$ bearing a  locally homogeneous affine connection of non Riemannian type is a torus.
\end{theorem}

 Recall first that a well known result  of J. Milnor shows that a compact surface bearing a flat affine connection is a torus~\cite{Milnor} (see also~\cite{Benzecri}).
 
 \begin{proof}  In the case of a non flat connection, theorem~\ref{imprimitive connections} shows that $M$ is locally modelled on an imprimitive homogeneous space. Then
 proposition~\ref{imprimitive} finishes the proof.
 \end{proof}

 We give now  the proof of theorem~\ref{projective}.

 \begin{proof} The starting point of the proof is the classification obtained in~\cite{BMM} of all possible Killing algebras of a two dimensional  locally homogeneous projective connection. Indeed, Lemma 3 and Lemma 4  in~\cite{BMM} prove that  eiher $\phi$ is flat, or  the Killing algebra of $\phi$ is the Lie algebra of  one of the following Lie groups: $Aff(\RR)$ or  $SL(2,\RR)$. Moreover, in the last case the  isotropy is generated by a   one parameter unipotent  subgroup.
 
 Assume, by contradiction, that  the Killing algebra of $\phi$  is that of $Aff(\RR)$. Then, by theorem~\ref{model},  $M$ is locally modelled on $Aff(\RR)$ and, by theorem~\ref{isotropie compacte}, $M$ has to be a quotient of $Aff(\RR)$
 by a uniform lattice. Or, $Aff(\RR)$ is not unimodular and,  consequently, doesn't admit  lattices: a contradicition.
 
 Assume, by contradiction,  that the Killing algebra  is that of $SL(2,\RR)$. Then,  by theorem~\ref{model}, $M$ is locally modelled on $SL(2,\RR)/I$, with $I$ a  one parameter unipotent subgroup in $SL(2, \RR)$.

  Equivalently, $I$ is  conjugated to $\left(  \begin{array}{cc}
                                                                1  &   b \\
                                                                 0     &  1 \\
                                                                 \end{array} \right)$, with  $b \in \RR$. The  homogeneous space $SL(2,\RR)/I$ is difffeomorphic  to $\RR^2 \setminus \{0 \}$ endowed with the linear action
                                                                 of $SL(2, \RR)$.

                                                                 Notice  that the action of $SL(2,\RR)$ on $SL(2,\RR)/I$ preserves a nontrivial vector field. The expression of
   this vector field  in linear coordinates $(x_{1}, x_{2})$  on $\RR^2 \setminus \{ 0 \}$  is  $\displaystyle x_{1}     \frac{\partial}{\partial x_{1} }  +  x_{2}  \frac{\partial}{\partial x_{2}}$,  which is the fundamental generator
   of the one parameter group of homotheties. The flow of this vector field doesn't preserve the standard volume form and has a nonzero constant divergence:  $\lambda=2$.

 Let $X$ be the corresponding  vector field induced on $M$ and 
$div (X)$ the divergence of  $X$ with respect to the volume form 
 $vol$ induced on  $M$ by the standard $SL(2,\RR)$-invariant volume form  of $\RR^2 \setminus \{ 0 \}$. Recall that, by definition, $L_{X} vol =div(X) \cdot vol,$ where  $L_{X}$ is the Lie derivative. Here
 $div(X)$ is  the  constant function  $\lambda$.
 
Denote by $\Psi^t$ the time  $t$ of the flow generated by $X$. We get  $(\Psi^t)^*vol=exp( \lambda  t) \cdot vol,$ for all $t \in \RR$. But 
the  flow of  $X$ has to preserve the global volume  $\int_{M}vol$. This implies $\lambda =0$: a contradiction.

It remains that $\phi$ is flat. 
 \end{proof}

  We terminate the section with the following.

   \begin{proposition} If $M$ is a compact surface endowed with a locally homogeneous rigid geometric structure  admitting  a  semi-simple  Killing 
 algebra of dimension $3$, then either  $M$ is globally isomorphic  to a rotation invariant geometric structure on the two-sphere (up to a double cover), or the Killing algebra of $\phi$ preserves a 
 hyperbolic metric on $M$ (and $M$ is of genus $g \geq 2$).
 \end{proposition}
 
 \begin{proof} By theorem~\ref{model}, $M$ is locally modelled on $G/I$, with $G$ a $3$-dimensional connected simply connected semi-simple Lie group and $I$ a closed one parameter subgroup in $G$.
 
 Up to isogeny, there are only two such $G$: $S^3$ and $SL(2,\RR)$~\cite{Kir}. If $G=S^3$ then  $I$ is compact and coincides with the stabilizer of a point under the standard $S^3$-action on $S^2$. Consequently, $G/I$ identifies with $S^2$ and the 
 $G$-action on $G/I$ preserves the canonical metric of the two-sphere.
 
 The developing map from $\tilde M$ to $G/I$ has to be a diffeomorphism (see theorem~\ref{isotropie compacte}). Consequently,   $M$ is a quotient of the sphere $S^2$ by a discrete subgroup of $G$
 acting by deck transformations. Since a nontrivial isometry of $S^2$ which is not $-Id$ always admits fixed points, this discrete subgroup has to be of order two. Up tp a double cover, $(M, \phi)$ is isomorphic to $S^2$
 endowed with a rotation invariant geometric structure.
 
 Consider now the case where $G=SL(2,\RR)$. Then $M$ is locally modelled on $G/I$, where $I$ is a closed one parameter subgroup in $SL(2, \RR)$. We showed in the proof of 
 theorem~\ref{projective} that $I$ is not conjugated to a unipotent subgroup. 
 
 Assume   now  that   $I$ is  conjugated to a one parameter semi-simple subgroup in $SL(2, \RR)$.  We prove  that this assumption yields  a contradiction. In order to describe the geometry of $SL(2,\RR)/I$,  consider the adjoint representation
 of $SL(2, \RR)$ into its Lie algebra $sl(2, \RR)$. This $SL(2, \RR)$-action preserves the Killing quadratic form $q$, which is a non degenerate Lorentz quadratic form.
 Choose $x \in sl(2, \RR)$ a vector of unitary $q$-norm and consider its orbit under the adjoint representation. This orbit identifies with our  homogeneous space
 $SL(2, \RR)/I$, on which the restriction of the Killing form induces a two dimensional  $SL(2, \RR)$-invariant  complete Lorentz  metric $g$  of  constant nonzero sectional curvature~\cite{Wolf}.
 
 We prove now that there is no {\it compact}  surface locally modelled on the previous  homogeneous space $SL(2, \RR)/I$.
 
 Observe that $x$ induces on $SL(2,\RR)/I$ a $SL(2, \RR)$-invariant vector field $X$ and $g(X, \cdot)$ induces on $SL(2, \RR)/I$ an invariant one form $\omega$. Remark  that 
 $d \omega$ is a volume form. Indeed, $d\omega(Y,Z)=-g(X, \lbrack Y, Z \rbrack)$, for all  $SL(2, \RR)$-invariant vector fields $Y,Z$ tangents to $SL(2, \RR)/I$.
 
 Assume, by contradiction,  that $M$ is locally modelled on $SL(2, \RR)/I$. Then $M$ inherits the   one form $\omega$ whose differential is a
 volume form. This is in contradiction with Stokes' theorem.  Indeed $\int_{M}d \omega = \int vol  \neq 0$, but $d \omega$
 is exact.
 
 It remains that $I$ is conjugated to a one parameter subgroup of rotations in $SL(2,\RR)$ and then the  $SL(2, \RR)$-action on $SL(2,\RR)/ I$ identifies with the action by homographies on the Poincar\'e's
 upper-half plane. This action preserves the hyperbolic metric. Therefore $M$ inherits a hyperbolic metric and, consequently,  its genus is $\geq 2$.
 \end{proof}

 \section{Dynamics of local Killing algebra} \label{section4}
 
 A  manifold $M$ bearing a geometric structure $\phi$ admits  a natural partition given by the orbits of the action of the Killing algebra of $\phi$. Precisely, two points
 $m_{1}, m_{2} \in M$ are in the same subset of the partition if $m_{1}$ can be reached from $m_{2}$ by flowing along a finite sequence of local Killing fields. 
A connected open set in $M$ where $\phi$ is locally homogeneous lies in the same subset of this partition.

The Gromov's celebrated stratification theorem~\cite{DG, Gro} which was used and studied by many authors~\cite{DG,Gro, Benoist2, CQ, Feres} roughly states that, if $\phi$ is rigid,
the subsets  of this partition are locally closed in $M$. We adapt here Gromov's proof and get a more precise result in the particular case, where $\phi$ is of Riemannian type.

 \begin{theorem} Let $M$ be a connected manifold endowed with a geometric structure of Riemannian kind $\phi$. Then the orbits of the Killing algebra of $\phi$ are closed.
  \end{theorem}
 
 \begin{cor}  \label{metric} If $\phi$ is locally homogeneous on an open dense set, then
 $\phi$ is locally homogeneous on $M$. 
 \end{cor}
 
 \begin{proof} Let $g$ be a Riemannian metric preserved by all Killing fields of $\phi$. Consider also the $g$-orthonormal frame bundle $\pi : B \to M$. Then $B$ is a principal
 sub-bundle of the frame bundle $R^1(M)$ with structure group $O(n, \RR)$.
 
 Gromov's proof shows that (for any rigid geometric structure) there exists an integer $s \in \NN$ such that two points of $M$ where the $s$-jet of $\phi$ is the same are in the same
 orbit of the Killing algebra~\cite{DG,Gro, Benoist2, CQ, Feres}.
 
 We consider exponential local coordinates with respect to $g$. For each element of $(m, b) \in B$ we get local exponential coordinates around the point $m \in M$  in which we  take the $k$-jet $ \phi^k$ of $\phi$. This gives a map
 $$\phi^k : B \to Z^k$$
 
 with values in the variety $Z^k$  of $k$-jets of $\phi$.
 
 By Gromov's proof,   orbits of the Killing algebra  are the connected components of the projections on $M$ (through $\pi$)  of the pull-back through $\phi^k$ of $O(n,\RR)$-orbits of $Z^k$ (see the arguments in~\cite{DG}, section 3.5). The $O(n,\RR)$-orbits
 of $Z^k$ being compact, their pull-back through $\phi^k$  in $B$ are saturated closed sets. By   compactness of the fibers, the projection to $M$ of a saturated closed
 subset in $B$ is a closed set. Since connected components in $M$ of  closed sets are also closed, we get that the orbits of the Killing algebra are closed.
 \end{proof}

 The corollary  was known for Riemannian metrics. Indeed, in~\cite{Tri} the authors proved that  Riemannian metrics  whose  all scalar invariants are constant are locally homogeneous (this is known to fail in the pseudo-Riemannian setting~\cite{BV}).
 
 In the real analytic realm this implies the following more precise:
 
 \begin{cor} \label{extension}  If $M$ and $\phi$ are real analytic and $\phi$ is locally homogeneous on a nontrivial  open set, then $\phi$ is locally homogeneous on $M$.
 \end{cor}
 
 \begin{proof} In the real analytic setting, Gromov's proof shows that away from a nowhere dense analytic subset $S$ in  $M$, the orbits of the Killing algebra are  connected components of  fibers of an analytic map of constant rank~\cite{Gro} (section 3.2). With our hypothesis,  the orbits of the Killing algebra are exactly connected components of $M \setminus S$. Consequently, they are open sets. One apply  now corollary~\ref{metric} and get  that   the orbits are also closed. Since $M$ is connected,  the Killing algebra admits exactly one
 orbit.
 \end{proof}
 
 In the previous results the compactness of the orthogonal group was essential.

  We  prove  now the following result for affine unimodular connections which are not (necessarily) of Riemannian kind:
 
 \begin{theorem} \label{affine connection} Let $M$ be a compact  connected real analytic  surface endowed with a real  analytic unimodular   affine connection $\nabla$. If $\nabla$ is locally homogeneous on a  nontrivial  open set in $M$, then $\nabla$ is locally homogeneous on $M$.
 \end{theorem}
 
 Recall that $\nabla$ is said {\it unimodular}  if  there exists a volume form on $M$ which is invariant by the parallel transport~\cite{AVL}. This  volume form is automaticaly preserved by any  local Killing field of $\nabla$. We prove first the following useful:
 
 \begin{lemma} \label{unimodularlemma} Let $\nabla$ be a unimodular analytic  affine connection on an analytic  surface $M$. Then the dimension of the isotropy algebra at a point of $M$ is $\neq 2$.
 \end{lemma}
 
 \begin{proof}  Assume by contradiction that the isotropy algebra $\mathcal{I}$ at a point $m \in M$ has dimension two. Consider a system of  local exponential coordinates at $m$ with respect to $\nabla$ and, for all $k \in \NN$  take the $k$-jet of $\nabla$ in these  coordinates. Any volume preserving linear isomorphism of $T_{m}M$ gives another system of local exponential coordinates at $m$, with respect to which we consider the $k$-jet of the connection. This gives an algebraic  $SL(2,\RR)$-action on the vector space $Z^k$ of $k$-jets of affine connections on $\RR^2$ admitting a trivial underlying $0$-jet~\cite{DG,Gro}.
 
 Elements of $\mathcal{I}$ linearize in exponential coordinates at $m$. Since they preserve $\nabla$, they preserve in particular the $k$-jet of $\nabla$ at $m$, for all $k \in \NN$.
 This gives an embedding of $\mathcal{I}$ in the Lie algebra of $SL(2,\RR)$ such that  the corresponding (two dimensional)   connected subgroup of $SL(2,\RR)$ preserves the $k$-jet of $\nabla$ at $m$ for all  $k \in \NN$. 
 
 Now we use the  fact that {\it the stabilizers of a linear  algebraic $SL(2,\RR)$-action are of dimension  $\neq 2$}. Indeed, it suffices to check this statement for irreducible
 linear representations of $SL(2, \RR)$ for which it is well-known that the stabilizer in $SL(2,\RR)$  of a nonzero element is one dimensional~\cite{Kir}.
 
 It follows that the stabilizer of the $k$-jet of $\nabla$ at $m$ is of dimension three and contains the connected component of identity in $SL(2, \RR)$. Consequently,  in exponential coordinates  at $m$, each element of the connected  component of the identity in $SL(2, \RR)$  gives rise  to a local linear vector field which preserves $\nabla$ (for it preserves all $k$-jets
 of $\nabla$). The isotropy algebra $\mathcal{I}$ contains a copy of the Lie algebra of $SL(2, \RR)$: a contradiction, since $\mathcal{I}$ is of dimension two.
 \end{proof} 
 
 Our proof   of theorem~\ref{affine connection} will need analyticity in another  essential way. We will make use of an extendibility  result  for local  Killing fields proved first for Nomizu in the Riemannian setting~\cite{Nomizu} and generalized then for rigid geometric structures by Amores et Gromov~\cite{Amores,Gro} (see also~\cite{CQ,DG,Feres}). This phenomena states roughly that a  local Killing field of a {\it rigid analytic} geometric structure can be extended along any curve in $M$. We then get a multivalued Killing field defined on all of $M$ or, equivalently, a global Killing field defined
 on the universal cover. In particular, the Killing algebra in the neighborhood of any point is the same (as long as $M$ is connected).

 As an application of the Nomizu's phenomena we give (compare with theorem~\ref{model}):
 
 \begin{theorem} Let $M$ be a compact simply connected real analytic  manifold admitting a real analytic locally homogeneous  rigid geometric structure. Then $M$ is isomorphic
 to a homogeneous space $G/I$ endowed with a $G$-invariant geometric structure.
 \end{theorem}

 \begin{proof}  Since $\phi$ is locally homogeneous and $M$ is simply connected and compact, the local transitive action of the Killing algebra extends to a global action of the associated simply connected Lie group $G$  (we need compactness to insure that vector fields on $M$ are complete). All orbits have to be open, so there is only one orbit: the
 action is transitive and $M$ is a homogeneous space.
 \end{proof}
 
Let's  go back now to the proof of theorem~\ref{affine connection}. As before, in the real analytic setting  Gromov's stratification theorem shows that the locally homogeneous open dense set has to be dense~\cite{Gro, DG}. Note also 
 that Nomizu's extension phenomena doesn't imply that the extension of a family of linearly independent Killing fields, stays linearly independent. In general, the extension of a
 localy transitive Killing algebra, fails to be transitive on a nowhere dense analytic subset $S$ in $M$.  The unimodular affine connection is locally homogeneous on each connected
 component of $M \setminus S$. \\

{ \it  We prove now that  $S$ is empty.} \\

 Assume by contradiction that $S$ is not empty. Then we have the following crucial:
 
 \begin{lemma}  \label{dim isotropy}
 (i) The Killing algebra $\mathfrak{g}$ of $\nabla$  has dimension two and the isotropy algebra at a point of  $S$ is one dimensional.
 
 (ii) $\mathfrak{g}$  is isomorphic to the Lie algebra of the affine group of the line.
 \end{lemma}
 
 \begin{proof}  (i) Since the Killing algebra admits a nontrivial open orbit in $M$, its dimension is $\geq 2$. Pick up a point $s \in S$ and consider the linear morphism  $ev(s): \mathfrak{g} \to T_{s}M$ which associates to an element $K \in \mathfrak{g}$ its value $K(s)$. The kernel of this 
 morphism is the isotropy $\mathcal{I}$ at $s$.  Since the $\mathfrak{g}$-action is nontransitive in the neighborhood of $s$, the range of $ev(s)$ is $\leq 1$. This implies
 that the isotropy at $s$ is of dimension at least  dim  $\mathfrak{g}-1$.
 
 Assume, by contradiction,  that the Killing algebra has dimension at least three. Then the isotropy at $s$ is of dimension at least two.
 By lemma~\ref{unimodularlemma},  this dimension  never equals   two. Consequently,  the isotropy
  algebra at $s \in S$ is three dimensional.  The isotropy  algebra contains then a copy of the Lie algebra of $SL(2,\RR)$ (see the proof of lemma~\ref{unimodularlemma}).
  
   The local action of $SL(2,\RR)$ in the neighborhood of $s$ is conjugated to the the standard linear action of  $SL(2,\RR)$
     on  $\RR^2$.  This action has two orbits: the point $s$ and $\RR^2 \setminus \{s \} $. The open orbit $\RR^2 \setminus \{s \} $ identifies with a homogeneous space
     $SL(2,\RR)/I$. Precisely, 
      the stabilizer  $I$  in $G=SL(2, \RR)$ of a nonzero vector  $x \in T_{s}M$  is conjugated to the following   one parameter  unipotent subgroup of 
     $SL(2, \RR)$:   $\left(  \begin{array}{cc}
                                                                1  &   b \\
                                                                 0     &  1 \\
                                                                 \end{array} \right)$, with  $b \in \RR$. The action of $SL(2,\RR)$ on $SL(2,\RR)/I$ preserves the induced flat torsion free affine connection coming from $\RR^2$.

  By proposition 8 in ~\cite{AM-K}, the only $SL(2,\RR)$-invariant affine connection on $SL(2,\RR)/I$ is the previous flat  torsion free connection. Another way to prove this
  result is to consider the difference of a  $SL(2,\RR)$-invariant connection with the standard one. We get a $(2,1)$-tensor on $SL(2,\RR)/I$ which is  $SL(2,\RR)$-invariant. Equivalently, we get a  $ad(I)$-invariant 
  $(2,1)$-tensor on the quotient of the  Lie algebra $sl(2,\RR)$ by the infinitesimal generator of $I$~\cite{AVL}. A straightforward computation shows that the tensor has to be  trivial. This  gives a different  proof of the unicity of a  $SL(2,\RR)$-invariant connection.

  By analyticity,   $\nabla$ is torsion free and flat on all of $M$. In particular, $\nabla$ is locally homogeneous on all of $M$. This is in contradiction with our assumption.

 (ii)  The Killing algebra is  two dimensional. Thus it coincides with
  $\RR^2$ or with the Lie algebra of the affine group.  Consider $K_{1}, K_{2}$ a basis of the local  Killing  algebra and extend $K_{1}, K_{2}$ along a topological disk reaching 
 a point $s$ in   $S$.
 
 Recall that $\nabla$ is unimodular and let $vol$ be the volume form associated to $\nabla$.

  In the  case where the Lie algebra is $\RR^2$, $vol(K_{1}, K_{2})$ is a nonzero constant  (for being invariant by the Killing algebra and thus constant on the locally homogenous open set). Hence the  Lie algebra acts transitively in the neighborhood of $s \in S$. This is a contradiction: $S$ is then empty and $\phi$ is locally homogeneous
  on all of $M$.  \end{proof}

  For the sequel,  let $K_{1}, K_{2}$ be  two  local Killing fields at $s \in S$  which span the Killing  algebra. We  assume, without loss of generality, that $K_{1}$ and $K_{2}$ verify the Lie bracket relation  $\lbrack K_{1}, K_{2} \rbrack =K_{1}$.
  
   Recall that $K_{1}, K_{2}$ don't vanish both at a point $s \in S$. Indeed, if not the isotropy at $s$ has dimension  $2$, which  is impossible by lemma~\ref{unimodularlemma}.

  Notice that  $vol(K_{1}, K_{2})$ is not invariant by the action of the Killing algebra (since
  the adjoint representation of $Aff(\RR)$ is nontrivial). But we still have:
  
  \begin{proposition} \label{local volume} The local function $vol(K_{1},K_{2})$ is constant on the orbits of the flow generated by $K_{1}$.
  \end{proposition}
  
  \begin{proof}  The adjoint action $ad(K_{1})$ on $\mathfrak{g}$ is nilpotent. In particular, the adjoint representation of the one parameter subgroup generated by $K_{1}$
  is unimodular. Consequenlty,  the local flow generated by $K_{1}$ preserves $vol(K_{1},K_{2})$.
  \end{proof}
  
   \begin{proposition} \label{dim1} $S$ is a smooth $1$-dimensional manifold  (diffeomorphic to a finite union of circles).
 \end{proposition}
 
 \begin{proof} By lemma~\ref{dim isotropy}, the isotropy 
    $\mathcal{I}$ at a chosen point $s \in S$  is one dimensional and the range of the map $ev(s)$ equals  one. In particular, the orbit of $s$ under the action of $\mathfrak{g}$
    is one dimensional. The image of $ev(s)$ coincides 
    with  $T_{s}S$.  Consequently, the $\mathcal{I}$-action on $T_{s}M$ preserves the line $T_{s}S$.

    This shows that  $\mathfrak{g}$ acts
  transitively on each connected component of $S$.   
  
   In particular, each connected component of  $S$ is a one dimensional smooth submanifold in $M$ (recall that $S$ is nowhere dense in $M$).

  Since  $M$ is compact, each connected component of $S$ is diffeomorphic to a circle.

  \end{proof}

  We also have:

    \begin{proposition}   \label{cadre}  (i) $M \setminus S$ is locally modelled on the affine group of the line endowed with a translation invariant connection. 
    
  (ii)   $M \setminus S$ admits a (nonsingular) $\mathfrak{g}$-invariant  foliation $\mathcal{F}_{1}$  by lines.
  
  (iii) $\mathcal{F}_{1}$ coincides with the kernel of a (nonsingular) closed $\mathfrak{g}$-invariant  one form~$\omega$.
  
  (iv) The leafs of $\mathcal{F}_{1}$ are closed (in $M \setminus S$). They are endowed with a $\mathfrak{g}$-invariant  translation structure.
  
  (v) The space of leafs of $\mathcal{F}_{1}$  is Hausdorff.
  
   \end{proposition}

  \begin{proof}  (i) This comes from the fact that  the $\mathfrak{g}$-action  on $M \setminus S$ is simply transitive.
  
  (ii) The Lie bracket relation  $\lbrack K_{1}, K_{2} \rbrack =K_{1}$  implies
 that the flow of $K_{2}$ normalizes the flow of $K_{1}$ and thus the foliation spaned by $K_{1}$  is $K_{2}$-invariant. Consequently, the foliation spaned by $K_{1}$ is $\mathfrak{g}$-invariant  and it defines a foliation $\mathcal{F}_{1}$ well defined on $M \setminus S$.
 
 (iii) We locally define the one form $\omega$ such that $\omega(K_{1})=0$ and $\omega(K_{2})=1$. Since the basis $K_{1}, K_{2}$ is well defined, up to a Lie algebra isomorphism
 (which necessarily preserves  the derivative algebra $\RR K_{1}$  and sends $K_{2}$ on $K_{2} + \beta K_{1}$, for $\beta \in \RR$), $\omega$ is globally defined on $M \setminus S$.
 Also by Lie-Cartan's formula~\cite{AVL} $d \omega(K_{1},K_{2})=-\omega(\lbrack K_{1}, K_{2} \rbrack)=0$. 
 
 Thus the foliation is transversaly Riemannian in the sense of~\cite{Molino}. Another way to see it, is to observe that the projection of the local Killing field $K_{2} + \beta K_{1}$
 on $TM/ T\mathcal{F}_{1}$ is well defined (it doesn't depend on $\beta$). It defines a transverse vector field $\tilde K_{2}$ and, consequently, $\mathcal{F}_{1}$ is transversaly Riemannian.
 
 (iv) By proposition~\ref{local volume} and by  the previous point, $vol(K_{1},K_{2})$ is a  well define (nonconstant) function on $M \setminus S$  which is constant  on the leafs of $\mathcal{F}_{1}$. Therefore the leafs are closed in $M \setminus S$.

 Moreover, the action of the derivative algebra $\RR K_{1}$  preserves each leaf of $\mathcal{F}_{1}$. Hence each   leaf inherits a $\mathfrak{g}$-invariant translation structure.
 
 (v) This is a consequence of (iii) and (iv) (see~\cite{Molino}, chapter 3, Proposition 3.7). The connected components of the space of leafs are parametrized by the flow of $\tilde K_{2}$.
  \end{proof}

   Denote by $Y \in \mathfrak{g}$ a generator of $\mathcal{I}$ and by $X \in \mathfrak{g}$ a Killing field such that $X(s)$ span $T_{s}S$.
 By applying an automorphism of the Lie algebra of the affine group we can assume that either $Y=K_{1}$ and $X=K_{2}$, or $Y=K_{2}$ and $X=K_{1}$.\\
 
 {\bf Case I: unipotent isotropy}\\
 
 Assume first  that $Y=K_{1}$ and $X=K_{2}$.\\
   
   We study the local situation in the neighborhood of $s \in S$. For this local analysis  we will also denote by $S$ the connected component of $s$ in $S$.
   
   \begin{proposition} \label{unipotent isotropy} The isotropy at any point of $S$ is unipotent.
   \end{proposition}
   
   \begin{proof}

    Because of the Lie bracket relation,  the isotropy  $\mathcal{I}$ at $s$ acts trivially on $T_{s}S$. This implies that its generator  $Y$ acts trivially on the unique geodesic passing
  through $s$ and tangent to  the direction $T_{s}S$. Consequently $Y$ vanishes on this geodesic which has to coincide locally to $S$ (for $S$ is the subset of $M$ where
  $\mathfrak{g}$ doesn't act freely).

  Since the $Y$-action  on $T_{s}M$  is trivial on $T_{s}S$ and volume preserving, it is unipotent. In local exponential coordinates  $(x,y)$ at $s$, $Y$ is linearized and so
  conjugated to the linear vector field $y \frac{\partial}{\partial x}$. In these coordinates $S$ identifies locally with $y=0$ and the time $t$ of  the flow generated by $Y$ is
   $(x,y) \to (x+ty,y)$.

   Notice that the flow of $Y$ preserves a unique foliation in the neighborhood of $s$, which is given by $dy=0$. We will see later that this rules out the case of a  connected
   components of $S$ with isotropy $Y=K_{2}$ (see  Step 2 in  the proof of Case II, below).
   \end{proof}
   
   \begin{proposition} \label{extension} $\mathcal{F}_{1}$ extends  to  a nonsingular foliation  with closed leafs defined on all of $M$. 
   \end{proposition}
   
   \begin{proof}  The foliation $\mathcal{F}_{1}$ spaned  by $Y$ extends on $M$ by adding the leafs  $S$.  Moreover,
    all the leafs of the foliation $\mathcal{F}_{1}$ are closed on $M$. Indeed, the first integral $vol(K_{1},K_{2})$ on $M \setminus S$ vanishes exactly on the extra leafs $S$. Therefore,
    $vol(K_{1},K_{2})$ defines a first integral on all of $M$.
    \end{proof}

  \begin{proposition}   \label{both extensions}
  (i)  Each connected component of $M \setminus S$ is diffeomorphic to a cylinder $\RR \times S^1$.

  (ii) The elements of  $\mathfrak{g}$ extend to  global Killing fields on each connected component of $M \setminus S$.
  \end{proposition}
  
  \begin{proof}  By proposition~\ref{extension}, the leafs of $\mathcal{F}_{1}$ are circles.
  
   (i)  The flow of the transverse vector field $\tilde K_{2}$ (see the proof of proposition~\ref{cadre}) acts transitively on each connected component of the space of leafs of $\mathcal{F}_{1}$. Let  $\rbrack a, b \lbrack$ be the maximal domain of definition of a integral curve of $\tilde K_{2}$. Then the corresponding  connected component of $M \setminus S$ is diffeomorphic to  $\rbrack a, b \lbrack  \times S^1$ (we will se further that the maximal domain of definition is $\RR$).
 
 (ii)  Choose a local determination of $K_{1}$ and extend it along a leaf  of $\mathcal{F}_{1}$.   Our local determination  changes  into
  $\alpha K_{1}$, with $\alpha \in \RR$. Also extending $K_{2}$ along the same leaf, we will  find $K_{2}+ \beta K_{1}$,
   with $\beta \in \RR$.  We get then two  multivalued vector fields defined on a small cylinder containing the chosen leaf.
   
  But  $vol(\alpha K_{1}, K_{2}+ \beta K_{1})=vol(K_{1},  K_{2} )$ is a nonzero constant   function on  our leaf (for it is locally constant by proposition~\ref{local volume}). This implies $\alpha=1$, so $ K_{1}$ is globally defined 
  (and univalued) on the small cylinder containing the leaf. Since the fundamental group of the corresponding connected component of $M \setminus S$ is generated by our leaf (as a simple closed curve), $K_{1}$ extends in a global Killing field on all of the corresponding
  connected component of $M \setminus S$. 
  
   This also proves that the projection of $K_{2}$ on $T \mathcal{F}_{1}$ extends in  a global vector field (which is not Killing but preserves each leaf of $\mathcal{F}_{1}$ and its
   translation structure: its local expression is $f K_{1}$, with $f$ a function defined  on the space of leafs). Since the kernel of this projection is also well defined globally, $K_{2}$ extends
   on the corresponding connected component of $M \setminus S$ as well.
   \end{proof}

We have now  that $X$ and $Y$ extend to  globally defined vector fields on any connected component of $M \setminus S$. We have seen  that the flow of $X$ 
 parametrized the space of leafs of $\mathcal{F}_{1}$. In particular, the positive and the negative orbit of any point in $M \setminus S$ acccumulate  on $S$.
 
  A  contradiction is obtained in the  following;

  \begin{proposition}   \label{negative orbits} The negative orbit  of  points  in $M \setminus S$ under the flow of $X$  (equivalently the positive orbit under the flow of $-X$) cannot accumulate on $S$.
\end{proposition}
  
  \begin{proof}

      Recall that the Lie bracket relation implies that the flow of $X$ contracts $Y$ exponentialy. Thus the action of the negative flow of $X$ expands $Y$. Therefore,  a negative orbit under the flow of $X$ of a point $m$  in $M \setminus S$ never accumulates on $S$ (for $Y$ vanishes on $S$). 
      \end{proof}

 Another way to get a contradiction is to prove the following:
 
 \begin{proposition} \label{both complete} $X$ and $Y$ are complete on $M \setminus S$.
 \end{proposition}
 
 \begin{proof}  By proposition~\ref{negative orbits}, the negative orbits of $X$ are complete since they stay in a compact subset of $M \setminus S$. The positive orbits of
 $X$ are complete because they stay in a compact subset of the maximal domain of definition of $X$ (for $X$ can be extended to an open neighborhood of $S$ as in the proof
 of point  (ii) in proposition~\ref{both extensions}).
 
 As for $Y$, its orbits lie  in the (compact) leafs of $\mathcal{F}_{1}$, so they are complete.
 \end{proof}

 Proposition~\ref{both complete} implies then that there is a   locally free transitive action of $Aff(\RR)$ on each connected component of $M \setminus S$. Consequently, each
 connected component of  $M \setminus S$ is diffeomorphic to a homogeneous space $Aff(\RR) / \Gamma$, where $\Gamma$ is a discrete subgroup of $Aff(\RR)$. 
 
 Since the $Aff(\RR)$-action preserves the finite volume of $M \setminus S$ given by $vol$, $\Gamma$ has to be a lattice in $Aff(\RR)$: a contradiction. \\

  {\bf Case II: semi-simple isotropy}\\
  
  We assume now that $Y=K_{2}$ and $X=K_{1}$.\\

 {\it We will show that the Killing algebra $\mathfrak{g}$ preserves two transverse foliations by lines $\mathcal{F}_{1}$ and $\mathcal{F}_{2}$ on $M$, with $\mathcal{F}_{2}$
 geodesic.}\\
 
 {\bf Step 1: Construction of $\mathcal{F}_{1}$}.   We prove that the  foliation $\mathcal{F}_{1}$ constructed in proposition \ref{cadre} extends to  all of $M$. 
 
 Assume first that $S$ is connected. In the neighborhood of $s \in S$  the local foliation spaned by the nonsingular vector field $K_{1}$ agrees with $\mathcal{F}_{1}$
 on $M \setminus S$. Since this stands  in the neigborhood of each point of $S$, the  foliation $\mathcal{F}_{1}$  extends to  a nonsingular foliation on $M$ by adding the  leaf $S$.   
      
   If  $S$ admits several connected components, the previous argument applies  in the neighborhood of each component of $S$.
   
   Notice that Poincar\'e-Hopf's theorem shows that $M$ is a torus.\\

  {\bf Step 2: Construction of $\mathcal {F}_{2}$.}  Since the action of $\mathcal{I}$ on $T_{s}M$ is unimodular  (for it preserves $vol(s)$) and contracts the direction 
  $T_{s}S$ (for $\lbrack Y, X \rbrack=-X$), this action (preserves and) expands a  unique line $\mathcal {F}_{2}(s) \in T_{s}M$. This constructs a smooth $\mathfrak{g}$-invariant  line field along $S$: the unique line field (preserved and) expanded by the local isotropy.
  
  In the neighborhood of $s$, consider the unique geodesic foliation $\mathcal {F}_{2}$ which extends the previous line field. It is the image of  the line field through the exponential map, with respect to $\nabla$,  along $S$. Since $\mathcal{F}_{2}$ is $\mathfrak{g}$-invariant, it extends to a  foliation (which we still denote by $\mathcal{F}_{2}$) on $M$  transverse to $\mathcal {F}_{1}$. We obtain then:
  
  \begin{proposition} \label{a Lorentz metric}$M$ admits a flat $\mathfrak{g}$-invariant lorentz metric.
  \end{proposition}
  
  \begin{proof} Let $v \in T_{m}M$ and consider $q(v)=vol(v_{1} ,v_{2}),$ where $v_{1},v_{2}$ are the components of $v$ under the splitting $$T_{m}M=T\mathcal{F}_{1}(m)\oplus T\mathcal {F}_{2}(m).$$ This constructs a $\mathfrak{g}$-invariant  lorentz metric on $M$. In particular, $q$ is locally homogeneous and hence of (lorentzian) constant  sectional curvature~\cite{Wolf}.
  
  Since $M$ is a torus, Gauss-Bonnet's theorem (see its Lorentzian  version in~\cite{Wolf})  implies that $q$ is flat (locally isomorphic to  $\RR^2$ endowed with $dxdy$)~\cite{Wolf}. 
  \end{proof}
  
  {\bf Step 3: Construction of a $\mathfrak{g}$-invariant  vector field $T$ on $M$ which vanishes exactly on $S$.}

  Recall that  the local vector field $Y$ generates the isotropy in  the neighborhood of $s\in S$.
  Let $T$ be  the projection of $Y$ on the second factor of the decomposition $$TM=T\mathcal {F}_{1} \oplus T\mathcal {F}_{2}.$$ Then
  $T$ is obviously $Y$-invariant. Since $X$ spans $\mathcal F_{1}$ and $X$ and $Y$ commutes modulo $X$, it follows that $T$ is also $X$-invariant. The vector field $T$ is $\mathfrak{g}$-invariant and vanishes on $S$. As before, $T$ is defined locally in the neighborhood of $s \in S$, but since $T$ is $\mathfrak{g}$-invariant, it extends to a global vector field on $M$.
  
 Note that   $\mathfrak{g}$ being  transitive on $M \setminus S$, the vector field  $T$ doesn't vanish on $M \setminus S$. We prove now:
  
  \begin{proposition} The flow of  $T$ preserves the volume form of the unimodular connection (and also the volume form of the flat lorentz metric $q$).
   \end{proposition}
  
  \begin{proof}  Since both $T$ and the volume form $vol$ are $\mathfrak{g}$-invariant, the divergence $div(T)$ is also $\mathfrak{g}$-invariant. Consequently, $div(T)$  is constant on each connected component of $M \setminus S$. The time $t$ of the $T$-flow acts then on $vol$ by multiplication with $exp(\lambda t)$, where $div(T)=\lambda  \in \RR$.
 
 But this $T$-action has to preserve the global (finite)  volume of each connected component of $M \setminus S$ . This implies $\lambda=0$ and, consequently,  $vol$ is $T$-invariant.
 
 Remark that the underlying  volume form of the lorentz metric being also $\mathfrak{g}$-invariant, it is a constant multiple of $vol$. Consequently, it is also $T$-invariant.
  \end{proof}

  We will get a contradiction by showing the following 
  
  \begin{proposition} The divergence of $T$ with  respect to the volume form of the flat lorentz metric  equals $1$.
  \end{proposition}
  
  \begin{proof} Choose a point $s$   in $S$. In exponential coordinates $(x,y)$, at $s$, with respect to the flat lorentz metric $q$, the foliations $\mathcal F_{1}$ and $\mathcal F_{2}$
  are generated by the isotropic directions $\frac{\partial}{\partial x}$ and $\frac{\partial}{ \partial y}$ of the standard lorentz metric $dxdy$.

   The isotropy $\mathcal I$ at $s$ is   generated by the linear vector field $Y=-x \frac{\partial}{\partial x} + y \frac{\partial}{ \partial y}$.

 It follows, by construction,  that $T$ equals $y \frac{\partial}{ \partial y}$. The divergence of $T$ with respect to $dx \wedge dy$ equals $1$. 
  \end{proof}
  
  \section{Locally homogeneous affine connections} \label{section 5}
  
    Recall that T. Nagano and K. Yagi completely classified torsion free flat affine connections (affine structures)  on the real two torus~\cite{Nagano}. They proved that, except a well described family of incomplete affine structures (see the  nice description given in~\cite{Benoist1}), all the other are homogeneous constructed from  faithful affine actions of $\RR^2$ on the affine
 plane $GL(2, \RR) \ltimes \RR^2 /  GL(2, \RR)$ admitting an open orbit. This induces a translation invariant affine structure on  $\RR^2$. The quotient
 of $\RR^2$ by  a lattice is a homogeneous affine two dimensional tori. 
 
 Notice that the previous homogeneous affine structures are complete if and only if the corresponding  $\RR^2$-actions on the affine plane are transitive.

 In particular, Nagano-Yagi's result prove  that a real two dimensional torus
 locally modelled on the affine space $GL(2, \RR) \ltimes \RR^2 /  GL(2, \RR)$ such that  the linear part of the holonomy morphism lies  in $SL(2, \RR)$ is always complete
 and homogeneous (the corresponding torsion free flat affine connection on the torus is translation invariant). 
 
 Here we prove:
 
 \begin{theorem}  \label{complete and homogeneous} Let $M$ be a compact surface locally modelled on a homogeneous space $G/I$ such that $G/I$ admits a $G$-invariant   affine connection $\nabla$
and a volume form. Then:
 
 (i)   the corresponding affine connection on $M$  is homogeneous (up to a finite cover), except if $\nabla$ is the Levi-Civita  connection of a hyperbolic metric 
 (and the genus of $M$ is $\geq 2$).
 
 (ii) the corresponding affine connection on $M$ is complete.
 \end{theorem}

 \begin{remark} In particular, a unimodular  affine connection on the two torus   has constant Christoffel symbols with respect to global translation invariant
 coordinates.
 \end{remark}

 \begin{proof} Note first that, in the case where $I$ is compact, theorem~\ref{isotropie compacte} proves that $M$ is complete. Moreover, if $M$ is of genus $0$ then the underlying
locally homogeneous  Riemannian metric has to be  of positive sectional curvature (by Gauss-Bonnet's theorem) and $G/I$ coincides with the standard sphere $S^2$ seen as a homogeneous space of $S^3$.
Since $M$ is simply connected, the developing map gives 
a global isomorphism with $S^2$.

Also if the genus of $M$ is one, the underlying locally homogeneous metric on $M$ has to be flat. It follows that $G/I$ coincides with the only homogeneous flat Riemannian space
$O(2, \RR) \ltimes \RR^2 / O(2, \RR)$. Bieberbach's theorem (see, for instance,~\cite{Wolf})  implies then that $M$ is homogeneous (up to a finite cover).

Surfaces of genus $g \geq 2$ cannot admit a homogeneous geometric structure locally modelled on $G/I$, with $I$ compact. Indeed, here the Riemannian metric induced on $M$  is of  constant negative curvature. It is well  know that the isometry group of a hyperbolic metric on $M$ is finite (hence it cannot act transitively)~\cite{Kobayashi}.\\

 {\it Consider now the case where $I$ is  noncompact}.\\
 
  Denote also by $\nabla$ the locally homogeneous affine connection induced on $M$. The Killing algebra $\mathfrak{g}$ acts
 locally on $M$ preserving $\nabla$ and a volume form $vol$. The proof of theorem~\ref{affine connection} implies  that $\mathfrak{g}$ is of dimension at most $3$.
 Therefore $G$ is of dimension $3$ and $I$ is a  one parameter subgroup in  $G$. Since $I$ is supposed noncompact, its (faithful) isotropy action   on $T_{o}G/I$, where $o$ is the origin point of $G/I$, identifies $I$ either with a semi-simple, or with a unipotent one parameter  subgroup in $SL(2,\RR)$.\\
 
 {\bf Case I: semi-simple  isotropy}\\ 
 
 The isotropy action on $T_{o}G/I$ preserves two line fields. Consequently, there exists on $G/I$ two $G$-invariant line fields. Since $G/I$ admits also a $G$-invariant volume form,
 there exists on $G/I$ a $G$-invariant Lorentz metric $q$ (see the construction in proposition~\ref{a Lorentz metric}).
 
 By Poincar\'e-Hopf's theorem $M$ is a torus and $q$ has to be flat (by Gauss-Bonnet's theorem). It follows that $G$ is the automorphism group  $SOL$ of the flat Lorentz
 metric $dxdy$ on $\RR^2$. By the Lorentzian version of Bieberbach's theorem, this structure is known to be complete and homogeneous (up to a finite cover of $M$)~\cite{CD}.\\

 {\bf Case II: unipotent isotropy}\\
 
 Here the isotropy action preserves a vector field in $T_{o}G/I$. This yields a $G$-invariant vector field $\tilde X$ on $G/I$ and  a $G$-invariant one form $\tilde{\omega}=vol(\tilde X, \cdot)$.
 We then have:
 
 \begin{proposition} $M$ admits  a $\mathfrak{g}$-invariant vector field $X$ and a $\mathfrak{g}$-invariant closed one form $\omega$ vanishing on $X$. Thus the foliation $\mathcal{F}$
 generated by $X$ is transversally Riemannian.
 \end{proposition}
 
 \begin{proof} Since  $d \tilde{ \omega}$ and $vol$ are $G$-invariant, there exists $\lambda \in \RR$ such that $d \tilde{ \omega}= \lambda vol$.

 Denote  by $\omega$  the  one form on $M$ associated to $\omega$ and by $X$ the vector field induced on $M$ by  $\tilde X$. It follows that $\omega$ vanishes on $X$. Also
 $\int_{M} d \omega= \lambda vol(M)$, where the volume of $M$ is calculated with respect to the volume form induced by $vol$. Stokes' theorem yield $\lambda =0$ and, consequently,
 $\omega$ is closed.
 \end{proof}
 
 \begin{proposition} The normal  subgroup $H$ of $G$ which preserves each leaf of the foliation $\tilde{\mathcal{F}}$ generated by $\tilde X$ on $G/I$  is  two dimensional and abelian.
  \end{proposition}
  
  \begin{proof} The group $G$ preserves the foliation $\tilde{\mathcal{F}}$ and its transverse  Riemannian structure. Therefore  an element of $G$ fixing one leaf of $\tilde{\mathcal{F}}$, will fixe all the leafs.
  The subgroup $H$ is nontrivial, since it contains the isotropy. The action of $G/H$ on the transversal of the foliation preserves a Riemannian structure, so it is of dimension
  at most one. Since this action has to be transitive, the dimension of $G/H$ is exactly one and the dimension of $H$ is two.
  
  Since $H$ preserves $\tilde X$, the elements of $H$ commutes in restriction to each leaf of $\tilde{\mathcal{F}}$. Hence $H$ is abelian.
  \end{proof}
  
  We will make use of the following result proved in~\cite{Dumitrescu} (pages 17-19):
  
  \begin{lemma}  $\tilde X$ is a central element in $\mathfrak{g}$. Consequently, $X$ is a global Killing field on $M$ preserved by $\mathfrak{g}.$
  \end{lemma}
  
  In order to prove that $M$ is homogeneous, we will construct  a second global Killing field on $M$ which is not tangent to the foliation $\mathcal F$ generated by $X$. For this we will study the holonomy group.

  Once again $M$ is of genus one, since it admits a nonsingular vector field. 
 Since the fundamental group of the torus is abelian, its  image $\Gamma$ by  the holonomy morphism is an abelian subgroup of $G$.

We prove  first that $\Gamma$ is nontrivial. Assume by contradiction that $\Gamma$ is trivial. Then the developing map is well defined on $M$ and we get a local diffeomorphism
$dev : M \to G/I$. The image has to be open and closed (for $M$ is compact). Hence, $dev$ is surjective. By Ehresmann's submersion theorem, $dev$ is a covering map. Remark~\ref{model simply connected} shows that $G/I$ can be considered simply connected, which implies $dev$ is a diffeomorphism: a contradiction, since $M$ is not simply connected.

  Considering  finite covers of $M$, the previous proof rules out the case where $\Gamma$ is finite. Consequently, $\Gamma$ is an infinite subgroup of $G$.

 Consider its  (real) Zariski closure  $\overline{\Gamma}$ in $G$, which is an abelian subgroup of positive dimension. This is possible, since by Lie's classification~\cite{Mos,Olver},
 $G$ is locally isomorphic to a real algebraic  group acting (transitively)  on $G/I$. Therefore we can assume  that $G$ is algebraic.

 Up to a finite cover of $M$, we can assume $\overline{\Gamma}$ connected (algebraic groups admit  at most finitely many connected components). 
 
 \begin{proposition}  \label{holonomie et champs de Killing} Any element $a$ of the Lie algebra of $\overline{\Gamma}$ defines a global Killing field on $M$. 
 
 Moreover, if the 
 one parameter subgroup  of $G$ generated by $a$  intersects $\Gamma$ nontrivially, then the orbits of the corresponding Killig field on $M$ are closed.
  \end{proposition}
 
 \begin{proof}  The action  of $\Gamma$ on $a$ (by adjoint representation) being trivial,   $a$ defines a  $\Gamma$-invariant vector field on $G/I$ which  descends on $M$.
 
 Assume now that $\Gamma$ intersects nontrivially the one parameter subgroup generated by $a$. One  orbit of the corresponding Killing field on $M$ develops in the model
 $G/I$ as one  orbit of the Killing field $a$. Since vector fields on compact manifolds are complete, the image of the developing map contains all of the orbit of $a$. In particular, the
 corresponding orbit of $a$  contains distinct points which are in the same $\Gamma$-orbit. Therefore, the corresponding orbit on $M$ is closed.
  \end{proof}
 
 We prove now:
 
 \begin{proposition} $\overline{\Gamma}$ is not a subgroup of $H$.
 \end{proposition}
 
 \begin{proof} Consider a one parameter subgroup in $G$,  generated by an element in the Lie algebra of  $\overline{\Gamma}$,  which nontrivially  intersects
 $\Gamma$. By proposition~\ref{holonomie et champs de Killing}, $a$ defines a global Killing  field $K$ with closed orbits on $M$.
 
  Assume by contradiction that $\overline{\Gamma}$ is  a subgroup of  $H$. It follows that the orbits of $K$ coincide with those of $X$. Consequently, the orbits of $X$ are closed and
  the space of leafs of $\mathcal F$ is a one dimensional manifold (see proposition~\ref{cadre}, point (v)). Since $M$ is compact,  the space of leafs is diffeomorphic to a circle $S^1$.
  
  Consider  the developing map $dev: \tilde{M} \to G/I$ of the $G/I$-structure. In particular, this is also the developing map of the transverse structure of the foliation $\mathcal{F}$.
  Since the holonomy group acts trivially on the transversal $\tilde T$ of $\tilde{\mathcal F}$, $dev$  descends to   a  local diffeomorphism  from the space of leafs of $\mathcal{F}$ (parametrized by  $S^1$) to $\tilde T$.

 Since $G/I$ is simply connected, the closed one form $\tilde{\omega}$ admits a primitive $\tilde f: G/I \to \RR$. Consequently, $\tilde f$ is a first integral for $\tilde{\mathcal F}$ and
  $\tilde f \circ dev$ descends to a local diffeomorphism from $S^1$ to $\RR$. This map has to be onto since the image is open and closed. We get a topological contradiction.
  \end{proof}
 
 By proposition~\ref{holonomie et champs de Killing}, any one parameter subgroup in $G$ generated by an element of the Lie algebra of   $\overline{\Gamma}$ non contained in the Lie algebra of  $H$ provides a global Killing field $K$ on $M$ such that
 the abelian group  generated by the flows of $X$ and of $K$ acts transitively on $M$. Therefore the $G/I$-structure on $M$ is homogeneous.

 (ii)  Consider $\tilde X, \tilde K$ the corresponding Killing vector fields on $G/I$. They generate a two dimensional  abelian   subgroup $A$  of $G$  acting with an open orbit
 on $G/I$.  Recall that $vol(\tilde X, \tilde K)$  is an  $A$-invariant function on $G/I$. Hence,  $vol(\tilde X, \tilde K)$  is a nonzero constant in restriction to the open orbit. By continuity, $vol(\tilde X, \tilde K)$     equals  the same  nonzero constant on the closure
 of the open orbit. This implies that $\tilde X, \tilde K$  remain  linearly independent on the closure of the open orbit, which implies that the open orbit is also closed. Since $G/I$ is connected, the open orbit is  all of $G/I$.
  \end{proof}
  
 The previous proof combined with a result  of~\cite{Dumitrescu}   leads to the following classification result:

 \begin{theorem} \label{classification} Let $M$ be a compact surface locally modelled on a homogeneous space $G/I$ such that $G/I$ admits a $G$-invariant   affine connection $\nabla$ and a volume form. 
 
 If  $I$ is noncompact, then $G$ is three dimensional and:

(i) either $G$ is  the unimodular $SOL$ group  and its action preserves a flat (two dimensional) Lorentz metric;

(ii) or $G$ is the Heisenberg group and its action preserves the standard flat torsion free affine connection on $\RR^2$ together with  a nonsingular closed one form  $\tilde{\omega}$  and with  a nontrivial  parallel vector field
  $\tilde X$ such that $\tilde{\omega}(\tilde X)=0$;
  
(iii) or  $G$ is isomorphic to the product $\RR \times Aff(\RR)$ and its action preserves the canonical bi-invariant torsion free and complete connection of $Aff(\RR)$ (for which the full
automorphism group is $Aff(\RR) \times Aff(\RR)$) together with a geodesic Killing field.
\end{theorem}

The model (i) was obtained in the previous proof in the case where the isotropy is semi-simple. The models (ii) and (iii) correspond to the case where  the isotropy is unipotent.
This  classification follows from~\cite{Dumitrescu} (see Proposition 3.5 and pages 15-19).

Theorem~\ref{complete and homogeneous}  can be deduced from theorem~\ref{classification}. Indeed, in the case (ii)  $G$ preserves  a flat torsion free affine connection and
 Nagano-Yagi's result applies. The completeness of  compact  surfaces locally modelled on the homogeneous space (iii) was proved in Proposition 9.3 of~\cite{Zeghib} (the homogeneity
 follows from the proof of Proposition 10.1 in~\cite{Zeghib}).


\begin{thebibliography}{AA}
          
\bibitem{AVL} {\it D. Alekseevskij, A. Vinogradov, V. Lychagin}, Basic Ideas and Concepts in Differential Geometry, E.M.S., Geometry I, Springer-Verlag (1991).       
          
 \bibitem{Amores} {\it A. M. Amores}, Vector fields of a finite type $G$-structure, J. Differential Geom., {\bf 14(1)}, (1979), 1-6.
 
 \bibitem{AM-K} {\it T. Arias-Marco, O. Kowalski}, Classification of locally homogeneous affine connections with arbitrary torsion on $2$-dimensional manifolds, Monatsh. Math.,
 {\bf 153}, (2008), 1-18.
 
 
 \bibitem{Benoist1} {\it Y. Benoist}, Tores affines, Contemp. Math., {\bf 262}, (2000), 1-32.
 
 
 
\bibitem{Benoist2} {\it Y. Benoist}, Orbites de structures rigides,  Integrable systems and foliations (Montpellier), Boston, Birka\"user, (1997).

\bibitem{BFL} { \it Y. Benoist, P. Foulon, F. Labourie}, Flots d'Anosov \`a distributions stables et instables diff\'erentiables, Jour. Amer. Math. Soc.,
{\bf 5}, (1992), 33-74.

\bibitem{BF} {\it J. Benveniste, D. Fisher}, Nonexistence of invariant rigid structures and invariant almost rigid structures, Comm. Annal. Geom., {\bf 13(1)}, (2005), 89-111.

\bibitem{Benzecri} {\it J. Benzecri}, Sur les vari\'et\'es localement affines et localement projectives, Bull. Soc. Math. France, {\bf 88}, (1960), 229-332.


 
 \bibitem{BMM}{\it R. Bryant, G. Manno, V. Matveev}, A solution of a problem of Sophus Lie: Normal forms of $2$-dim metrics admitting two projective vector fields, Math. Ann., {\bf 340(2)}, (2008), 437-463.
 
 \bibitem{BV}  {\it P. Bueken, L. Vanhecke}, Examples of curvature homogeneous Lorentz metrics, Class. Quant. Grav., {\bf 5}, (1997), 93-96.
 
 \bibitem{CQ} {\it A. Candel, R. Quiroga-Barranco}, Gromov's centralizer theorem, Geom. Dedicata {\bf 100}, (2003), 123-155.
 
 \bibitem{CD} {\it Y. Carri\`ere, F. Dalbo}, G\'en\'eralisations du premier th\'eor\`eme de Bieberbach sur les groupes cristalographiques, Enseign. Math., {\bf 35(2)}, (1989), 245-263.
 
\bibitem{DG} {\it G. D'Ambra, M. Gromov}, Lectures on transformations groups: geometry and dynamics, Surveys
in Differential Geometry (Cambridge), (1990), 19-111.

 \bibitem{Dumitrescu} {\it S. Dumitrescu},  Dynamique du pseudo-groupe des isom\'etries locales sur une vari\'et\'e lorentzienne analytique de dimension $3$, Ergodic Th. Dyn. Systems, {\bf 28(4)}, (2008), 1091-1116.

\bibitem{Egorov} {\it I. Egorov}, On the order of the group of motions of spaces with affine connection, Dokl. Akad. Nauk. SSSR, {\bf 57}, (1947), 867-870.

\bibitem{Feres} {\it R. Feres}, Rigid geometric structures and actions of semisimple Lie groups, Rigidit\'e, groupe fondamental et dynamique, Panorama et synth\`eses,
{\bf 13}, Soc. Math. France, Paris, (2002).

\bibitem{Tri}  {\it P. Friedbert, F. Tricerri, L. Vanhecke}, Curvature invariants, differential operators and local homogeneity, Trans. Amer. Math. Soc., {\bf 348},
(1996), 4643-4652.


\bibitem{Gro}  {\it M. Gromov}, Rigid transformation groups, G\'eom\'etrie Diff\'erentielle, (D. Bernard et Choquet-Bruhat Ed.), Travaux en cours,
Hermann, Paris, {\bf 33}, (1988), 65-141.

\bibitem{Kir} {\it A. Kirilov}, El\'ements de la th\'eorie des repr\'esentations, M.I.R., (1974).


\bibitem{Kobayashi} {\it S. Kobayashi}, Transformation groupes in differential geometry, Springer-Verlag, (1972).

\bibitem{K}  {\it O. Kowalski}, Counter-example to the second Singer's theorem, Ann. Global Anal. Geom., {\bf 8(2)}, (1990), 211-214.


\bibitem{KOV} {\it O. Kowalski, B. Opozda, Z. Vl\'asek}, A classification of locally homogeneous connections on $2$-dimensional manifolds via group-theoretical approach,
CEJM, {\bf 2(1)}, (2004), 87-102.

\bibitem{LT} {\it  F. Lastaria  \& F. Tricceri}, Curvature-orbits and locally homogeneous Riemannian manifolds, Ann. Mat. Pura Appl., {\bf 165(4)}, (1993), 121-131.

\bibitem{Lie} {\it S. Lie}, Theorie der Transformationsgruppen, Math. Ann., {\bf 16}, (1880), 441-528.

\bibitem{Melnick} { \it K. Melnick}, Compact Lorentz manifolds with local symmetry, Journal of Diff. Geom., {\bf 81 (2)}, (2009), 355-390.


\bibitem{Milnor} {\it J. Milnor}, On the existence of a connection of curvature zero, Comment. Math. Helv., {\bf 32}, (1958), 215-223.

\bibitem{Molino}  {\it P. Molino}, Riemannian Foliations, Birkhauser, (1988).

\bibitem{Mos}  {\it G. Mostow}, The extensibility of local Lie groups of transformations and groups on surfaces, Ann. of Math., {\bf 52(2)}, (1950), 606-636.

\bibitem{Nagano} {\it T. Nagano, K. Yagi}, The affine structures on the real two torus, Osaka J. Math., {\bf 11}, (1974), 181-210.



\bibitem{Nomizu}    {\it K. Nomizu}, On local and global existence of Killing vector fields, Ann. of Math. (2), 
{\bf 72}, (1960), 105-120.

\bibitem{Olver} {\it P. Olver}, Equivalence, invariants and symmetry, Cambridge Univ. Press., Cambridge, (1995).

\bibitem{Opozda} {\it B. Opozda}, Locally homogeneous affine connections on compact surfaces, Proc. Amer. Math. Soc., {\bf 9(132)}, (2004), 2713-2721.



\bibitem{Sharpe} {\it R. Sharpe}, Differential Geometry, Cartan's Generalization of Klein's Erlangen Program, Springer, (2000).


\bibitem{Wolf}  {\it J. Wolf}, Spaces of constant curvature, McGraw-Hill Series in Higher Math., (1967).


\bibitem{Zeghib} {\it A. Zeghib}, Killing fields in compact Lorentz $3$-manifolds, J. Differential  Geom., {\bf 43}, (1996), 859-894.

\end{thebibliography}
\end{document}